\documentclass[aos,preprint,dvipsnames]{imsart}

\RequirePackage[OT1]{fontenc}
\RequirePackage{graphicx}

\RequirePackage{amsthm,amsmath,amsfonts,amssymb}

\RequirePackage[cmex10]{amsmath}
\RequirePackage[round]{natbib}
\RequirePackage[colorlinks,citecolor=blue,urlcolor=blue]{hyperref}
\usepackage[english]{babel}   
\usepackage{graphicx}
\usepackage{epstopdf}
\usepackage{amsfonts}
\usepackage{amsmath,amssymb}		
\usepackage{bbm} 	
\usepackage{amsthm}
\usepackage{xcolor}
\usepackage{accents}
\usepackage{easy-todo}
\usepackage{mathtools}
\usepackage{paralist} 
\RequirePackage{natbib}

\newtheorem{prop}{Proposition}
\newtheorem*{prop*}{Proposition}

\newtheorem*{theorem*}{Theorem}

\newtheorem*{assumption*}{Assumption}

\numberwithin{equation}{section}

\newcommand{\rN}{\mathbb{N}}             

\newcommand{\prob}{\mathbb{P}}

\newcommand{\E}{\mathcal{E}}                     

\newcommand{\Field}{\mcB}

\newcommand{\mcB}{\mathcal{B}}

\newcommand{\thetab}{\boldsymbol\theta}

\newcommand{\Thetab}{\boldsymbol\Theta}

\newcommand{\FieldSub}{{\cal F}}
\newcommand{\expect}{{\mathrm E}}

\makeatletter
\newcommand*\rel@kern[1]{\kern#1\dimexpr\macc@kerna}
\newcommand*\widebar[1]{%
  \begingroup
  \def\mathaccent##1##2{%
    \rel@kern{0.8}%
    \overline{\rel@kern{-0.8}\macc@nucleus\rel@kern{0.2}}%
    \rel@kern{-0.2}%
  }%
  \macc@depth\@ne
  \let\math@bgroup\@empty \let\math@egroup\macc@set@skewchar
  \mathsurround\z@ \frozen@everymath{\mathgroup\macc@group\relax}%
  \macc@set@skewchar\relax
  \let\mathaccentV\macc@nested@a
  \macc@nested@a\relax111{#1}%
  \endgroup
}
\makeatother

\definecolor{dkgreen}{rgb}{0.1,0.6,0.1}

\graphicspath {{figures/}}

\startlocaldefs
\endlocaldefs

\begin{document}

\begin{frontmatter}
\title{
 On Bounded Completeness \\ and 
 the $L_1$-Denseness  of Likelihood Ratios}

\ \\ \today

\begin{aug}
\author{\fnms{Marc} \snm{Hallin,\!}\thanksref{m1}\ead[label=e1,email]{mhallin@ulb.ac.be}}
\author{\fnms{Bas J.M.} \snm{Werker,\!\!\!}\thanksref{m2}
\ead[label=e2,email]{b.j.m.werker@tilburguniversity.edu}}
\and
\author{\fnms{Bo} \snm{Zhou}\thanksref{m3}\ead[label=e3,email]{bzhou@vt.edu}}
\runauthor{Hallin, Werker, and Zhou.}
\address{
\thanksmark{m1}Universit\'{e} libre de Bruxelles, Belgium, \thanksmark{m2}Tilburg University, the Netherlands, and \thanksmark{m3}Virginia Tech, Blacksburg, VA }
\end{aug}

\begin{abstract}
The classical concept of bounded completeness and its relation to sufficiency and ancillarity play a fundamental role in unbiased estimation, unbiased testing, and the validity of inference in the presence of nuisance parameters. In this short note, we  provide a direct proof of a  little-known result by \cite{Far62} on a characterization of bounded completeness   based 
 on an $L^1$ denseness property of the linear span of likelihood ratios.  As an application, we show that an experiment with infinite-dimensional observation space is boundedly complete iff suitably chosen restricted   subexperiments with finite-dimensional observation spaces are.  
\end{abstract}

\end{frontmatter}

 \section{Introduction} The concept of (bounded) completeness, in connection with sufficiency and ancillarity, is playing a major role at the foundations of statistical inference: unbiased point estimation (the so-called Lehmann-Scheff\' e Theorem), unbiased testing (similarity and Neyman $\alpha$-structure), and statistical decision in the presence of nuisance parameters. These subjects have  generated an abundant literature, some controversy, and several puzzles that remain unsettled or undecidable up to now. 
 
 The first explicit appearances of the concept of completeness are in \cite{LehSchf47, LehSchf50, LehSchf55}, who coined  the terminology. The   complex interrelations between sufficiency, completeness, and ancillarity  are the main research interest of R.R. Bahadur's early career  publications: see \cite{Bahad52, Bahad54, Bahad55a, Bahad55b, Bahad57, Bahad79}, in close interaction with other major contributions such as   \cite{basu1955, basu1958, basu1959,KandT75}.  
 A  complete bibliography of the subject is impossible here; we refer to \cite{Leh81,Stig92, Ghosh02}, or \cite{Ghoshetal.10} for references and comments, to \cite{Stig01} for a historical perspective. 
 
Completeness and bounded completeness bear a clear relation to   functional analysis\footnote{Citing  \cite{MandelRusch87}, ``it would be very surprising if,  when coining the term {\em complete}, Lehmann and Scheff\' e did not have in mind the definition of a {\it complete subset} in a Hilbert space $H$, namely,  {\em A subset $C\subset H$ is {\it complete in $H$} if the only element $h\in H$ which is orthogonal to all elements of $C$ is~$h=0$.''}}  which, however,  seldom has been exploited. Among the few exceptions are results by \cite{Far62}, \cite{MandelRusch87}, \cite{Plachky77}, \cite{PlachRukh91}. \cite{Far62}, in particular,  provides an interesting characterization of bounded completeness based on an $L^1$ denseness property of the linear span of likelihood ratios. His result, however, appears at the end of a paper on algebras of functions,  as an informal comment on a more general statement (Corollary 4) 
most statisticians  would find hard to interpret, and with proofs  involving unfamiliar  algebraic developments.  In this short note, we are providing a simple statement, with a direct and elementary proof based on Mazur's Theorem, of a slightly more general version of this result. That version allows for a useful characterization of bounded completeness based on bounded completeness of subexperiments.  As an application, we show how this provides a straightforward proof of the bounded completeness of Brownian drift experiments.   
  
  \section{A characterization of bounded completeness}\label{sec2} Let 
\begin{equation}\label{defexperiment}
{\cal E}\coloneqq \left({\cal X}, {\cal B}, {\cal P}\coloneqq\{{\rm P}_{\thetab} \vert\, \thetab\in\Thetab\}\right)
\end{equation}
 denote a statistical experiment with observation space  
 $({\cal X},{\cal B})$    and distribution family $\cal P$ indexed by a parameter $\thetab$ with values in some (possibly infinite-dimensional)   
parameter space~$\Thetab$.   
Assu\-ming that all distributions~${\rm P}_{\thetab}\in{\cal P}$ are    absolutely continuous with respect to some {\it privileged dominating measure}\footnote{A privileged measure exists as soon as the family $\cal P$ is a dominated one (see \cite{HS49} or Chapter~4 in \cite{Rusch14});   it has the property that, denoting by~$B$ an element of $\cal B$, ${\rm P}_{\boldsymbol 0}(B)=0$ iff~${\rm P}_{\thetab}(B)=0$ for all $\thetab\in\Thetab$.} ${\rm P}_{\boldsymbol 0}$ (which may or may not belong to $\cal P$), denote by~${\rm dP}_{\thetab}/{\rm dP}_{\boldsymbol 0}$ the Radon-Nikodym derivative of~${\rm P}_{\thetab}$ with respect to ${\rm P}_{\boldsymbol 0}$ and consider the set 
\begin{equation}\label{defS}{\cal S}\coloneqq \text{{\rm span}}\big(\left\{ {\rm dP}_{\thetab}/{\rm dP}_{\boldsymbol 0}:\, \thetab\in\Thetab
 \right\} \big)
\end{equation}
 of all  linear combinations of $\cal E$'s likelihood ratios.\footnote{This is a (minor) abuse of terminology: actually, ${\rm dP}_{\thetab}/{\rm dP}_{\boldsymbol 0}$ is a {\it likelihood ratio} if ${\rm P}_{\boldsymbol 0}\in {\cal P}$, a {\it likelihood} if~${\rm P}_{\boldsymbol 0}\notin {\cal P}$.} Obviously, denoting by ${\rm E}_{\thetab}$ the expectation under  ${\rm P}_{\thetab}$, ${\rm E}_{\boldsymbol 0}({\rm dP}_{\thetab}/{\rm dP}_{\boldsymbol 0})=1$ for all $\thetab$, so that ${\cal S}$ is a subset of  the space $L^1({\cal B}, {\rm P}_{\boldsymbol 0})$ of all integrable $\cal B$-measurable random variables $Y$ equipped with the norm $\Vert Y\Vert_1\coloneqq {\rm E}_{\boldsymbol 0}(\vert Y\vert)$. 
 
 Next consider a sub-$\sigma$-field ${\cal B}_0\subseteq{\cal B}$. Associated with that sub-$\sigma$-field is a   subexperi\-ment~$\E_{{\cal B}_0}\coloneqq \left({\cal X},{\cal B}_0,{\cal P}
 \right)$  
 of $\E$   with the ${\cal B}_0$-measurable  likelihood ratios ${\rm E}_{\boldsymbol 0}\left[{\rm dP}_{\thetab}/{\rm dP}_{\boldsymbol 0}\left\vert\right. 
{\cal B}_0 \right]$. Call $\E_{{\cal B}_0}$ a {\it restricted subexperiment} of $\E$. Similar to~\eqref{defS}, define 
\begin{equation}\label{defSB0}
{\cal S}_{{\cal B}_0}\coloneqq 
\text{{\rm span}}
\big(\left\{ 
{\rm E}_{\boldsymbol 0}\left[{\rm dP}_{\thetab}/{\rm dP}_{\boldsymbol 0}\left\vert\right. {\cal B}_0\right]: \thetab\in\Thetab
 \right\} \big).
\end{equation}

 We say that the $\sigma$-field ${\cal B}_0$ (equivalently, the restricted subexperiment\footnote{\cite{Far62} only considers the ${\cal B}_0={\cal B}$ case ({\it ``full''} bounded completeness).} $\E_{{\cal B}_0}$)  is {\it (boundedly) complete} in $\cal E$ if any~${\cal B}_0$-measurable (bounded\footnote{For simplicity, we throughout  use ``bounded" for ``essentially bounded;'' recall that a random variable $Y$ is essentially bounded if there exists a constant $0\leq M < \infty$ such that ${\rm P}_{\thetab} (\vert Y\vert \leq M) =1$ for all $\thetab\in\Thetab$. })  random variable $Y$ such that ${\rm E}_{\thetab}(Y)=0$ for all $\thetab\in\Thetab$ is  ${\rm P}_{\thetab}$-a.s.\  equal to zero for all $\thetab\in\Thetab$ (equivalently, ${\rm P}_{\boldsymbol 0}$-a.s.); the family $\cal P$ or the experiment $\cal E$ itself are called (boundedly) complete if~$\cal B$ is (boundedly) complete in $\cal E$. \medskip 
 
 Denote by $L^1({\cal B}_0, {\rm P}_{\boldsymbol 0})$   the space of  ${\rm P}_{\boldsymbol 0}$-integrable  ${\cal B}_0$-measurable real-valued functions and by $L^\infty({\cal B}_0, {\rm P}_{\boldsymbol 0})$ the space of ${\rm P}_{\boldsymbol 0}$-a.s.\ bounded ${\cal B}_0$-measurable ones: clearly, ${\cal S}_{{\cal B}_0}$ is a subset of~$L^1({\cal B}_0, {\rm P}_{\boldsymbol 0})$. We then have the following characterizations, which extends to restricted subexperiments the corresponding result of \cite{Far62}.
 
 \begin{prop}\label{prop1}  The sub-$\sigma$-field ${\cal B}_0\subseteq{\cal B}$ (equivalently, the restricted experiment $\E_{{\cal B}_0}$) is boundedly complete in $\cal E$ iff ${\cal S}_{{\cal B}_0}$ is dense in~$L^1({\cal B}_0, {\rm P}_{\boldsymbol 0})$.
  \end{prop}
 
 The proof of this proposition relies on the so-called Mazur Theorem or Mazur Lemma (see, e.g.,  \cite{Bre83}  page 38 or \cite{EkTem76} page  389) which, in this context, takes the following form.
 
  \begin{theorem*}{\hspace{-2.5mm}{\rm (Mazur)}}  Let $\cal M$ be a vector subspace in some topological vector space\footnote{A normed linear space  is a topological vector space for the topology induced by its norm. } $\cal V$. Let~$\cal K$ denote an open nonempty convex subset  of $\cal V$ such that ${\cal K}\cap {\cal M}=\emptyset$. Then, there exists a closed hyperplane ${\cal H}\subseteq {\cal V}$ that contains $\cal M$ but is disjoint from~$\cal K$.
      \end{theorem*}
 
 \begin{proof} (i) Assume that ${\cal S}_{{\cal B}_0}$ is dense in~$L^1({\cal B}_0, {\rm P}_{\boldsymbol 0})$. Let $Y$ 
  be (essentially) bounded with 
 \begin{equation}\label{EY=0}
 {\rm E}_{\thetab}(Y)=0\quad\text{ for all $\thetab\in\Thetab$}. \end{equation}
 This $Y$ determines a continuous linear functional $f_Y$ on~$L^1({\cal B}_0, {\rm P}_{\boldsymbol 0})$  (an element of the~$L^\infty({\cal B}_0, {\rm P}_{\boldsymbol 0})$ dual of~$L^1({\cal B}_0, {\rm P}_{\boldsymbol 0})$), mapping~$Z\in L^1({\cal B}_0, {\rm P}_{\boldsymbol 0})$ to~$f_Y(Z) \coloneqq {\rm E}_{\boldsymbol 0}(YZ)$. Under~\eqref{EY=0}, that functional vanishes on~${\cal S}_{{\cal B}_0}$: indeed, for any $\thetab$, 
 $$ {\rm E}_{\boldsymbol 0}\left[Y{\rm E}_{\boldsymbol 0}\left[{\rm dP}_{\thetab}/{\rm dP}_{\boldsymbol 0}\left\vert\right. {\cal B}_0\right]\right] =  {\rm E}_{\thetab}(Y) =0.$$  Since ${\cal S}_{{\cal B}_0}$ is dense in~$L^1({\cal B}_0, {\rm P}_{\boldsymbol 0})$, this continuous functional vanishes on all of $L^1({\cal B}_0, {\rm P}_{\boldsymbol 0})$.  Now, assume that ${\rm P}_{\thetab}(Y\neq  0)>0$ for some $\thetab$: then,  ${\rm P}_{\boldsymbol 0}(Y\neq  0)>0$ and,   in view of~\eqref{EY=0},   the indicator $Z^+_Y$ of $[Y> 0]$ cannot be ${\rm P}_{\thetab}$-a.s.\ zero.   Since  $Y$  is~${\cal B}_0$-measurable, $Z^+_Y\in~\!L^1({\cal B}_0, {\rm P}_{\boldsymbol 0})$. Clearly,~$f_Y(Z_Y)= {\rm E}_{\boldsymbol 0}(YZ_Y)>0$, which contradicts the fact that $f_Y(Z)=0$ for  all~$Z$ in~$L^1({\cal B}_0, {\rm P}_{\boldsymbol 0})$. Thus, $Y=0$ ${\rm P}_{\thetab}$-a.s. for all $\thetab$ and ${\cal B}_0$ is complete.  
  
 (ii) Conversely, let us assume that ${\cal S}_{{\cal B}_0}$ is not dense in~$L^1({\cal B}_0, {\rm P}_{\boldsymbol 0})$. Letting ${\cal V}=L^1({\cal B}_0, {\rm P}_{\boldsymbol 0})$ and~${\cal M}={\cal S}_{{\cal B}_0}$, it follows from Mazur's theorem that there exists a nonempty  ball  ${\cal K}$ in~$L^1({\cal B}_0, {\rm P}_0)$ that does not intersect ${\cal M}={\cal S}_{{\cal B}_0}$ and a closed hyperplane  $\cal H$   such that ${\cal S}_{{\cal B}_0}\subseteq {\cal H}$ and ${\cal K}\cap {\cal H}=\emptyset$.  Hyperplanes in ${\cal V}=L^1({\cal B}_0, {\rm P}_{\boldsymbol 0})$ are of the form
\begin{equation}\label{calH}{\cal H}=\{Z\in L^1({\cal B}_0, {\rm P}_0) : f_Y(Z):= {\rm E}_{\boldsymbol 0}[YZ]=0\}\end{equation}
for some  (essentially bounded)   $Y\in   L^1({\cal B}_0, {\rm P}_{\boldsymbol 0})$ which is not ${\rm P}_{\boldsymbol 0}$-a.s.\ zero (else, \eqref{calH} would not be a hyperplane). Hence, there exists~$Y\in   L^1({\cal B}_0, {\rm P}_{\boldsymbol 0})$ such that $f_Y$ 
vanishes  on~${\cal S}_{{\cal B}_0}$ but not on the ball ${\cal K}$.   Letting $Z_0= {\rm E}_{\boldsymbol 0}\left[{\rm dP}_{\thetab}/{\rm dP}_{\boldsymbol 0}\left\vert\right. {\cal B}_0\right]$, we have~$Z_0\in{\cal S}_{{\cal B}_0 }$ hence
  $$0=f_Y(Z_0)= {\rm E}_{\boldsymbol 0}(YZ_0)={\rm E}_{\boldsymbol 0}(Y{\rm E}_{\boldsymbol 0}\left[{\rm dP}_{\thetab}/{\rm dP}_{\boldsymbol 0}\left\vert\right. {\cal B}_0\right])={\rm E}_{\thetab}(Y)$$  for any $\thetab\in\Thetab$. Since  $Y$ is not ${\rm P}_{\boldsymbol 0}$-a.s.\ zero,~${\cal B}_0$ is not boundedly complete in $\E$.   \end{proof}
 
 \section{From  bounded complete subexperiments to full bounded completeness} In this section, we demonstrate the value of Proposition~\ref{prop1} as a tool for establishing further properties about completeness, sufficiency, and ancillarity. 
    As in Section~\ref{sec2}, consider   subexperiments $\E_{{\cal B}_0}$ of $\E$, ${\cal B}_0\subseteq{\cal B}$. The following result  is quite helpful in proving bounded completeness in the case of an infinite-dimensional parameter space $\Upsilon$. 

Recall that a sequence~$\left(\FieldSub_k\right)_{k\in\rN}$ of~$\sigma$-fields  is a \emph{convergent filtration} in $\Field_{0}$ if $\FieldSub_k\!\subseteq~\!\!\FieldSub_{k+1}\!\!\subset~\!\!\Field_{0}$ for each~$k\in\rN$ and~$\bigcup_{k\in{\mathbb N}}\FieldSub_k=~\!\Field_{0}$. Using the same notation as in Section~\ref{sec2}, we then have the following result. 

\begin{prop}\label{lem:CompletenessInfinite} Consider the experiment $\E$ in \eqref{defexperiment}. 
Denoting by $\left(\FieldSub_k\right)_{k\in\rN}$  a convergent filtration in $\Field_{0}\subseteq {\cal B}$ satisfying ${\cal S}_{\FieldSub_k}\subseteq {\cal S}_{\Field_{0}}$ for all~$k\in~\!\rN$. Then~$\E_{{\cal B}_0}$ is boundedly complete in $\E$ iff each restricted subexperiment $\E_{\FieldSub_k}$ is.
%
\end{prop}

\begin{proof}
	The ``only if" parts of the proposition is  trivial.  Turning to the  ``if" part, let us establish that if $\E_{{\cal F}_k}$ is boundedly complete in $\E$ for all~$k$, then~${\cal S}_{\Field_0}$ is dense in~$L^1(\Field_0, {\rm P}_0)$ which, in view of Proposition~\ref{prop1}, implies that~${\E}_{\Field_0}$ is boundedly complete in~$\E$. Let $Y\in~\!L^1({\Field}_0,{\rm P}_{\boldsymbol 0})$. From Doob's martingale convergence theorem (\cite{Doob53}, Chapter VII Theorem 4.3), we know that~$Y_k\coloneqq\expect_{\boldsymbol 0}\left[Y\vert\FieldSub_k\right]$ converges to~$Y$ in~$L^1({\Field}_0,{\rm P}_{\boldsymbol 0})$. Namely, for any~$\delta >~\!0$, there exists $K=K(\delta)$ such that~${\rm E}_{\boldsymbol 0}\vert  Y_k - Y\vert\leq\delta/2$ for~$k\geq K$. Con\-sider~$Y^\prime \coloneqq Y_K$ (an element of $L^1({\Field}_0,{\rm P}_{\boldsymbol 0})$). Since~${\E}_{{\cal F}_k}$ is boundedly complete in $\E$, it follows from Proposition~\ref{prop1} that ${\cal S}_{{\cal F}_k}$ is dense in $L^1({\cal F}_k,{\rm P}_{\boldsymbol 0})$; hence, for any $\delta >0$, there exists  some~$Y^{\prime\prime}\in {\cal S}_{{\cal F}_k}$ such that~${\rm E}_{\boldsymbol 0}\vert Y^{\prime\prime} - Y^{\prime} \vert \leq\delta/2$. Therefore,
	\begin{equation}\label{close}{\rm E}_{\boldsymbol 0}\vert Y^{\prime\prime} - Y \vert \leq {\rm E}_{\boldsymbol 0}\vert Y^{\prime\prime} - Y^{\prime}\vert + {\rm E}_{\boldsymbol 0}\vert Y^{\prime} - Y\vert\leq \delta. 
\end{equation}
Now, ${\cal S}_{{\cal F}_k}\subseteq {\cal S}_{{\cal B}_0}$. Hence, for all $Y\in L^1(\Field_0, {\rm P}_0)$, this $Y^{\prime\prime}$ also is an element of~${\cal S}_{{\cal B}_0}$. In view of~\eqref{close}, it  is arbitrarily close, in the~$L^1(\Field_0, {\rm P}_0)$ sense, to $Y$, and ${\cal S}_{\Field_0}$ thus is dense in~$L^1(\Field_0, {\rm P}_0)$.  The  claim   follows from  applying Proposition~\ref{prop1} again.\end{proof}


\newcommand{\sampleSpace}{{\cal X}}
\newcommand{\LRSpace}{{\cal S}}
\newcommand{\dee}{\mbox{d}}
\newcommand{\ExpO}{{\rm E}_{\boldsymbol 0}}
\section{Bounded completeness of Brownian drift experiments}\label{sec:BrownianDrift}

\noindent As an illustration of the above results, consider the Brownian drift experiment $\E_{\text{\tiny{\rm drift}}}$ defined on the sample\linebreak space~$\sampleSpace=C[0,1]$, the set of real-valued continuous functions on the unit interval, endowed with the $\sigma$-field $\Field$ generated by the topology of uniform convergence. The parameter space is  the set~$\Thetab=D[0,1]$ of {\it c\`adl\`ag} functions on $[0,1]$. For $\thetab\in\Thetab$ the probability measure ${\rm P}_{\thetab}$ is defined by $X(0)=0$ and
\begin{equation}\label{eqn:BrownianDrift}
	\dee X(u) = \thetab(u)\dee u + \dee W(u),
\end{equation}
where $W$ is a standard Brownian motion on $\left(\sampleSpace,\Field\right)$.

To show, via Proposition~\ref{lem:CompletenessInfinite}, that this experiment is boundedly complete,   consider    restricted experiments  
 under which  $X$ is observed on a finite grid of $u$ values  only. For fixed~$k\in{\mathbb N}$, define the $(2^k+1)$ grid points $u_{ik}=i2^{-k}$, $i=0,\ldots,2^k$, and the $2^k$ increments 
$$\Delta X(u_{ik}):= X(u_{ik})-X(u_{i-1,k}),\quad i=1,\ldots,2^k.$$ Restricted experiments $\E_{{\cal F}_k}$ are obtained via the sub-$\sigma$-fields  
$$\FieldSub_k:=\sigma\left(\Delta X(u_{ik}):~i=1,\ldots,2^k\right)=\sigma\left(X(u_{ik}):~i=0,\ldots,2^k\right) \subset \cal B.$$
 Note that $\FieldSub_k\subset\FieldSub_{k+1}$ and $\bigcup_k\FieldSub_k=\Field$; linearly interpolating the restricted observations~$X(u_{ik})$, $i=0,\ldots,2^k$, indeed, yields a uniform approximation of $X$ in $C[0,1]$.

Each of these restricted experiments  $\E_{\FieldSub_k}$ is a Gaussian shift experiment under which the~$2^k$-dimensional  random vector $\left(\Delta X(u_{ik})\right)_{i=1,\ldots,2^k}$  of  increments is observed. That vector is Gaussian 
 with unit covariance  matrix and   unrestricted mean in~${\mathbb R}^{2^k}\!$,  hence is complete (a fortiori, boundedly complete). Finally, $\LRSpace_{\FieldSub_k}\subset\LRSpace_\FieldSub$ since $\ExpO\left[\dee\prob_{\thetab}/\dee\prob_{\bf 0}\vert\FieldSub_k\right]=\dee\prob_{\tilde\thetab}/\dee\prob_{\bf 0}$ for
\begin{equation}
	\tilde\thetab(u)=\int_{u_{i,k-1}}^{u_{ik}}\thetab(v)\dee v,\quad u_{i,k-1}\leq u<u_{ik}.
\end{equation}
Bounded completeness of $\E_{\text{\tiny{\rm drift}}}$ thus follows from Proposition~\ref{lem:CompletenessInfinite}. \medskip

To provide a link to the literature on stochastic processes, remark that completeness of the Brownian motion experiment   also could be derived from the fact that stochastic exponentials of step functions are   \emph{total}; see Lemma~3.1 in Chapter~V of \citet{RevYor99}. Details are left to the reader.

\bibliographystyle{imsart-nameyear}
\bibliography{references}

\end{document}